\documentclass[11pt]{amsart}
\usepackage{amsfonts,amssymb,amscd,amsmath,enumerate,verbatim,calc,latexsym,pstcol,pst-plot,pst-3d,}

\def\NZQ{\Bbb}               

\def\ZZ{{\NZQ Z}}

\def\frk{\frak}               

\def\mm{{\frk m}}

\def\Phi{{\frk n}}
\def\Phi{{\frk N}}
\def\MR{{\mathcal R}}

\def\opn#1#2{\def#1{\operatorname{#2}}} 
\opn\chara{char} \opn\length{\ell} \opn\pd{pd} \opn\rk{rk}
\opn\projdim{proj\,dim} \opn\injdim{inj\,dim} \opn\rank{rank}
\opn\depth{depth} \opn\grade{grade} \opn\height{height}
\opn\embdim{emb\,dim} \opn\codim{codim}

\opn\Tr{Tr} \opn\bigrank{big\,rank}
\opn\superheight{superheight}\opn\lcm{lcm}
\opn\trdeg{tr\,deg}
\opn\reg{reg} \opn\lreg{lreg} \opn\ini{in} \opn\lpd{lpd}
\opn\size{size} \opn\sdepth{sdepth}
\opn\link{link}\opn\fdepth{fdepth}\opn\lex{lex}
\opn\div{div} \opn\Div{Div} \opn\cl{cl} \opn\Cl{Cl}
\opn\Spec{Spec} \opn\Supp{Supp} \opn\supp{supp} \opn\Sing{Sing}
\opn\Ass{Ass} \opn\Min{Min}\opn\Mon{Mon}
\opn\Ann{Ann} \opn\Rad{Rad} \opn\Soc{Soc}
\opn\Im{Im} \opn\Ker{Ker} \opn\Coker{Coker} \opn\Am{Am}
\opn\Hom{Hom} \opn\Tor{Tor} \opn\Ext{Ext} \opn\End{End}
\opn\Aut{Aut} \opn\id{id}

\opn\nat{nat}
\opn\pff{pf}
\opn\Pf{Pf} \opn\GL{GL} \opn\SL{SL} \opn\mod{mod} \opn\ord{ord}
\opn\Gin{Gin} \opn\Hilb{Hilb}\opn\sort{sort}
\opn\aff{aff} \opn\con{conv} \opn\relint{relint} \opn\st{st}
\opn\lk{lk} \opn\cn{cn} \opn\core{core} \opn\vol{vol}
\opn\link{link} \opn\star{star}\opn\lex{lex}\opn\set{set}
\opn\gr{gr}

\opn\dstab{dstab}
\opn\astab{astab}
\def\pot#1#2{#1[\kern-0.28ex[#2]\kern-0.28ex]}

\opn\dirlim{\underrightarrow{\lim}}
\opn\inivlim{\underleftarrow{\lim}}
\let\union=\cup

\let\Union=\bigcup

\def\Implies{\ifmmode\Longrightarrow \else
        \unskip${}\Longrightarrow{}$\ignorespaces\fi}
\def\implies{\ifmmode\Rightarrow \else
        \unskip${}\Rightarrow{}$\ignorespaces\fi}
\def\iff{\ifmmode\Longleftrightarrow \else
        \unskip${}\Longleftrightarrow{}$\ignorespaces\fi}

\let\:=\colon
\newtheorem{Theorem}{Theorem}[section]
\newtheorem{Lemma}[Theorem]{Lemma}
\newtheorem{Corollary}[Theorem]{Corollary}

\newtheorem{Remark}[Theorem]{Remark}

\let\epsilon\varepsilon
\let\kappa=\varkappa
\def\qed{\ifhmode\textqed\fi
      \ifmmode\ifinner\quad\qedsymbol\else\dispqed\fi\fi}
\def\textqed{\unskip\nobreak\penalty50
       \hskip2em\hbox{}\nobreak\hfil\qedsymbol
       \parfillskip=0pt \finalhyphendemerits=0}
\def\dispqed{\rlap{\qquad\qedsymbol}}

\opn\dis{dis}
\def\pnt{{\raise0.5mm\hbox{\large\bf.}}}

\opn\Lex{Lex}

\numberwithin{equation}{section}

\begin{document}

\title {Depth stability of edge ideals}

\author {J\"urgen Herzog and  Takayuki Hibi}

\address{J\"urgen Herzog, Fachbereich Mathematik, Universit\"at Duisburg-Essen, Campus Essen, 45117
Essen, Germany} \email{juergen.herzog@uni-essen.de}

\address{Takayuki Hibi, Department of Pure and Applied Mathematics, Graduate School of Information Science and Technology,
Osaka University, Toyonaka, Osaka 560-0043, Japan}
\email{hibi@math.sci.osaka-u.ac.jp}


\begin{abstract}
Let $G$ be a connected finite simple graph and let $I_G$ be the edge ideal of $G$. The smallest number $k$ for which $\depth S/I_G^k$ stabilizes is denoted by $\dstab(I_G)$. We show that $\dstab(I_G)<\ell(I_G)$ where $\ell(I_G)$ denotes the analytic spread of $I$. For trees we give a stronger upper bound for $\dstab(I_G)$. We also show for any two integers $1\leq a<b$ there exists a tree for which $\dstab(I_G)=a$ and $\ell(I_G)=b$.
\end{abstract}
\maketitle
\section*{Introduction}
It is a general feature that various homological and algebraic properties stabilize for powers of ideals in a Noetherian ring $R$. Most famous are the results of Brodmann who showed in \cite{Br1} that $\Ass(R/I^k)$ stabilizes for large $k$, and in \cite{Br2} that for $R$ a Noetherian local ring, $\depth R/I^k$ is constant  for all $k\gg 0$. Both statements are valid as well when $I$ is a graded ideal in the polynomial ring $S=K[x_1,\ldots,x_n]$, where $K$ is a field.

The natural question arises whether there exists a bound $k_0$ independent of $I$ but only dependent of $R$ with the property that $\Ass(R/I^k)$ and $\depth R/I^k$ are stable for all $k\geq k_0$. In \cite{HRV} the following invariants were introduced:
\[
\astab(I)=\min\{k\:\; \Ass(R/I^k)=\Ass(R/I^l)\text{ for all $l\geq k$}\},
\]
and when $R$ is local,
\[
\dstab(I)=\min\{k\:\; \depth R/I^k=\depth R/I^l\text{ for all $l\geq k$}\}.
\]
For graded ideals in the polynomial ring $S$, one defines  $\astab(I)$ and $\dstab(I)$ in the same way.

In general $\astab(I)$ and $\dstab(I)$ are incomparable. Indeed, in \cite{HRV} an example of a monomial ideal is  given with $\astab(I)<\dstab(I)$ and another one  with $\dstab(I)<\astab(I)$. On the other hand, it is shown in \cite[Theorem 4.1]{HQ} that $\astab(I), \dstab(I)<\ell(I)$ for any polymatroidal ideal $I\subset S$, where $\ell(I)$ denotes the analytic spread of $I$,  which by definition is the Krull dimension of $\MR(I)/\mm\MR(I)$. Here $\MR(I)$ denotes the Rees ring of $I$ and $\mm$ the graded maximal ideal $(x_1,\ldots,x_n)$ of $S$. Since $\ell(I)\leq \dim S$ for any graded ideal, we see that  $\astab(I)$ and $\dstab(I)$ are bounded by the Krull dimension of $S$ for any polymatroidal ideal $I$. There is no example of a graded ideal $I\subset S$ known to us for which $\astab(I)$ or $\dstab(I)$ exceeds the Krull dimension of $S$. Based on experimental evidence we therefore expect that $\astab(I), \dstab(I)< \ell(I)\leq \dim R$  when $R$ is a regular local ring. By passing to the completion, one would then have similar inequalities in the graded case.

In  this paper we show that $\dstab(I_G)< \ell(I_G)$ holds for the  edge ideal of a connected finite simple graph $G$. This is the result of Theorem~\ref{upper}. In the case that $G$ is a  tree we give  a stronger upper bound for  value of $\dstab(I_G)$ in terms of the data of the tree. Indeed, we show in Theorem~\ref{tree} that if $G$ is a tree with $n$ vertices and $m$ free vertices, that is, vertices of degree 1, then for the edge ideal $I_G$ of $G$ we have $\dstab(I_G)\leq n-m$. This upper bound is reached for example, when $G$ is a path graph. As a simple consequence of this result it can be shown (see Corollary~\ref{ab}) that for any two integers $a$ and $b$ with $1\leq a<b$, there exists a tree $G$ with $\dstab(I_G)=a$ and $\ell(I)=b$.

It is remarkable that the same number $n-m$ is a  bound of $\astab(I_G)$ when $G$ is a non-bipartite graph, as shown by Chen, Morey and Sung \cite{CMS}. In the bipartite case, one has $\astab(I_G)=1$ by a classical result of Simis, Vasconcelos and Villarreal \cite{SVV}.
Thus in any case $\astab(I)< \ell(I)$. In a very recent paper \cite{LT} Lam and Trung describe explicitly $\Ass(S/I_G^t)$ for all $t$ and any finite simple graph $G$ in terms of the ear decompositions of the induced strongly non-bipartite graphs of $G$. As a consequence, Lam and Trung succeed to give an explicit formula \cite[Corollary 4.8]{LT} for  $\astab(I_G)$ when $G$ is non-bipartite.

The proofs of the results in this paper depend very much on certain upper and lower bounds for the depth of powers of monomial ideals. The first of these bounds  is a formula \cite[Theorem 3.3]{HQ} giving an upper bound for the depth of $S/I^k$ when $I$ is monomial ideal generated in s single degree. This bound is expressed in terms of the relation graph of $I$. The details are given in Section~\ref{socle}. The other tool used to show that $\dstab(I_G)\ell<(I_G)$ is a result of Eisenbud and Huneke  \cite[Proposition 3.3]{EH} from which it follows that if $\MR(I)$ is Cohen--Macaulay, then $\dstab(I) =\min\{k\:\; \depth R/I^k=\dim R-\ell(I)\}$. Finally in the proof of Theorem~\ref{tree} we use a  result of Morey \cite[Corollary 3.7]{M} which says that if $G$ is a tree whose diameter is $\delta$,  then $\depth S/I_G^k\geq \min\{\left\lceil\frac{\delta-k+q}{3}\right\rceil, 1\}$. Here   $q$ is the number of vertices $v$ of $G$ with the property that $v$ is not  a free vertex and there is at most one vertex in the neighborhood of $v$ which is not free.

\section{An upper bound of  depth stability}
\label{socle}

We recall from \cite{HQ} the definition of  the linear relation graph of a monomial ideal $I$ generated in a single degree.  Let $G(I)=\{u_1,\ldots,u_m\}$ denote the unique minimal set of monomial generators of $I$. The {\em linear relation graph} $\Gamma$ of $I$ is the graph with edge set
\[
E(\Gamma)=\{\{i,j\}\: \text{there exist $u_k,u_l\in G(I)$ such that $x_iu_k=x_ju_l$}\}
\]
and vertex set $V(\Gamma)=\Union_{\{i,j\}\in E(\Gamma)}\{i,j\}$.

If $G$ is a finite graph, and $I_G$ is the edge ideal of $G$ we denote by $\Gamma_G$ the linear relation graph of $I_G$.

\begin{Lemma}
\label{relgra}
Let $G$ be a finite connected graph on $[n]$, and let $r$ be the number of vertices of $\Gamma_G$  and $s$  the  number of its connected components.
 \begin{enumerate}
\item[{\em (a)}] If $G$  bipartite and not a star graph, then $r=n$ and $s=2$.

\item[{\em (b)}] If $G$ is non-bipartite, then $r=n$ and $s=1$.
\end{enumerate}
\end{Lemma}

\begin{proof}
First note that $\{i,j\}\in E(\Gamma_G)$ if and only if there exists $k\in [n]$ such that $\{i,k\},\{j,k\}\in E(G)$. It follows that $i$ and $j$ belong to the same connected component of $\Gamma_G$ if and only if there is path of even length connecting $i$ and $j$.

(a) Let $U\union V$  be the decomposition of the vertex set of $G$. Since $G$ is not a star graph, we have $|U|,|V|\geq 2$. Hence, since $G$ is connected, for any $i\in U$ there exist $k\in V$ and $j\in U$ with $i\neq j$ such that $\{i,k\},\{j,k\}\in E(G)$. Therefore,  $i\in V(\Gamma_G)$. Similarly, if $i\in V$ then $i\in V(\Gamma_G)$. Let $i,j\in[n]$. Since $G$ is bipartite, there exists a path of even length connecting $i$ and $j$, if and only if either $i,j\in U$ or $i,j\in V$. Hence $\Gamma_G$ has two components, each of which is a complete graph.

(b) Since $G$ is non-bipartite,  $G$ contains at least one odd cycle, and since $G$ is connected it follows that any two vertices $i$ and $j$ can be connected by a path. Hence $r=n$ and $s=1$.
\end{proof}

We will use Lemma~\ref{relgra} to prove

\begin{Theorem}
\label{upper}
Let $G$ be a finite connected graph on $[n]$  with at least two edges. Then $\dstab(I_G)< \ell(I)$. In particular, $\dstab(I_G)<n$.
\end{Theorem}

\begin{proof}
For the proof of the theorem we will use the following two results from \cite{HQ}: let $S$ be a polynomial ring in $n$ variables  over a field,  $I\subset S$  a monomial ideal generated in a single degree, and let $\Gamma$ be the linear  relation graph of $I$. Suppose that $\Gamma$ has $r$ vertices and $s$ connected components. Then
\begin{eqnarray}
\label{aj}
\depth S/I^t\leq n-t-1 \quad\text{for} \quad t=1,\ldots,r-s,
\end{eqnarray}
and
\begin{eqnarray}
\label{hq}
\ell(I)\geq r-s+1.
\end{eqnarray}

We first assume that $G$ is non-bipartite. Since, by Lemma~\ref{relgra},   $r=n$ and $s=1$,  it follows from (\ref{aj}) that $\depth S/I_G^{n-1}=0$. Since for any edge ideal $I_G$ one has $\Ass(I_G^k)\subset \Ass(I_G^{k+1})$ for all $k$, c.f.~\cite[Theorem 2.15]{MM}, we see that $\depth S/I^k =0$ for all $k\geq n-1$. On the other hand, by (\ref{hq}) one has  $\ell(I)\geq n$ which implies that $\ell(I_G)=n$. Thus  we conclude that $\dstab(I_G)<\ell(I)$.

Next we assume that $G$ is bipartite.  The Rees ring $\mathcal{R}(I_G)$ of $I_G$ may be considered as the edge ring of the graph $G^*$ where $G^*$ is obtained from $G$ by adding a new vertex and connecting this vertex with all edges of $G$. Since $G$ is bipartite it follows from \cite[Corollary 2.3]{OH} that  $\mathcal{R}(I_G)$ is normal. Since  $\mathcal{R}(I_G)$ is a toric ring, a theorem of Hochster \cite[Theorem~1]{Ho} implies that $\mathcal{R}(I_G)$ is Cohen-Macaulay. Now we apply the result of Eisenbud and Huneke \cite[Proposition 3.3]{EH} which says that $\depth S/I_G^k\geq n-\ell(I)$ and $\depth S/I_G^k= n-\ell(I)$ for all $k\gg 0$, and whenever $\depth S/I_G^k=n-\ell(I)$ then $\depth S/I_G^l=n-\ell(I)$ for all $l\geq k$.  By (\ref{hq}) we have  $\ell(I_G)\geq n-1$. Since $G$ is bipartite, it follows that the rank of  the vertex-edge incidence matrix is $<n$. Since this rank gives us the analytic spread of $I_G$ we see that $\ell(I_G)=n-1$, so that $\depth S/I_G^k=1$ for all $k\gg 0$. On the other hand, assuming that $G$ is not a star graph,   Lemma~\ref{relgra} and (\ref{aj}) imply  $\depth S/I_G^{n-2}\leq 1$,  and hence $\depth S/I_G^k=  1$ for all $k\geq n-2$. Thus we see that $\dstab(I)\leq n-2=\ell(I)-1$ in this case. Now assume that $G$ is a star graph with center 1. Then $I_G=x_1(x_2,x_3,\ldots,x_n)$, and it follows that $\depth S/I_G^k=1$ for all $k$. Therefore, $\dstab(I_G)=1<\ell(I)$ if $G$ is not just an edge.
\end{proof}

\begin{Remark}{\em
\label{aleqb}
We notice that $\astab(I_G)\leq \dstab(I_G)$ if $G$ is bipartite,  and $\dstab(I_G)\leq \astab(I_G)$ if $G$ is non-bipartite. Indeed, in the first case, $\astab(I_G)=1$, since by a theorem of Simis, Vasconcelos and Villarreal \cite[Theorem 5.9]{SVV}, $I_G$ is normally torsionfree, which means that $\Ass(I_G^k)=\Min(I_G^k)$ for all $k$, and this  implies that $\Ass(I_G)=\Ass(I_G^k)$ for all $k\geq 1$. On the other hand, if $G$ is non-bipartite, then, as seen in the proof of Theorem~\ref{upper}, $\dstab(I)=\min\{k\: \depth S/I_G^k=0\}$. Thus,  $\dstab(I_G)=\min\{k\:\; \mm\in \Ass(S/I_G^k)\}$, and this implies that $\dstab(I_G)\leq \astab(I_G)$

The inequality, $\dstab(I_G)\leq \astab(I_G)$ together with \cite[Corollary 4.3]{CMS} give another proof of Theorem~\ref{upper}, and even a stronger bound for $\dstab(I)$ in the  non-bipartite case.
}
\end{Remark}

\section{The depth stability of trees}
\label{edge}

\begin{Theorem}
\label{tree}
Let $G$ be a tree on $[n]$ with $m$ leaves.
\begin{enumerate}
\item[{\em (a)}] $\dstab(I_G)\leq n-m$.
\item[{\em (b)}] Let $P$ be a path of maximal length in $G$, and suppose that all vertices of $G$ have distance at most $2$ to $P$. Then $\dstab(I_G)=n-m$.

\end{enumerate}
\end{Theorem}

\begin{proof}
(a) We may assume that $1$ is a  free vertex of $G$, and $2$ is its  unique neighbor. For simplicity, we set $I=I_G$, and show that $x_1-x_2$ is a non-zerodivisor of $S/I^k$ for all $k$. It is enough to prove this for $k=1$, because, as shown in \cite{SVV},  the edge ideal of any bipartite graph is normally torsionfree, equivalently $I^k$ has no embedded prime ideals for any $k\geq 1$.

We define the following $\ZZ^{n-1}$-grading on $S$ by setting $\deg x_1=\deg x_{2}=\epsilon_1$ and $\deg x_i=\epsilon_i$ for all $i>2$. Here $\epsilon_i$ is the $i$th canonical unit vector of $\ZZ^n$. Then $I$ is a $\ZZ^{n-1}$-graded ideal and $x_1-x_2$ is homogeneous of degree $\epsilon_1$. We need to show that $I:(x_1-x_2)\subseteq I$. Let $g\in I:(x_1-x_2)$. Because of the choice of our grading, we may assume that $g$ is homogeneous, which implies that $g=(x_1-x_2)^ax_1^{b_1}x_2^{b_2}\cdots x_n^{b_n}$ with $a,b_1,\ldots,b_n\geq 0$. It follows that $f=(x_1-x_2)^{a+1}x_1^{b_1}x_2^{b_2}\cdots x_n^{b_n}\in I$. Since $I$ is a monomial ideal this is the case if and only if for $i=0,\ldots,a+1$,  each monomial $v_i=x_1^{(a+1)-i+b_1}x_2^{i+b_2}x_3^{b_3}\cdots x_n^{b_n}$ belongs to $I$. Since $x_1x_2\in I$ and since $x_1x_2$ divides $v_i$ for $i=1,\ldots,a$ we see that $f\in I$ if and only if $v_0$ and $v_{a+1}\in I$. By the same reasoning it follows that $g\in I$ if and only if $w_0= x_1^{a+b_1}x_2^{b_2}\cdots x_n^{b_n}$ and $w_a= x_1^{b_1}x_2^{a+b_2}\cdots x_n^{b_n}$ belong to $I$. If $b_1>0$ and $b_2>0$, then obviously, $g\in I$. If either $b_1=0$ and $b_2>0$ or $b_1>0$ and $b_2=0$, then $x_ix_j$ with $i<i$ must divide $x_2^{b_2}\cdots x_n^{b_n}$, and hence $x_ix_j$  divides $w_0$ and $w_a$. Finally, if $b_1=b_2=0$, then $x_ix_j$ with $i<j$ divides $x_3^{b_3}\cdots x_n^{b_n}$, and hence $x_ix_j$  divides $w_0$ and $w_a$.

Next we show that $\depth S/I^{n-m}=1$. Because of \cite[Proposition 3.3]{EH},  this  will then imply that $\dstab(I)\leq n-m$. In order to prove  $\depth S/I^{n-m}=1$,  we show that $\depth S/J=0$ where $J=(I^{n-m},x_1-x_2)$. Let $E$ be the set of edges of $G$ which are not leaves. Since $G$ is tree,  it follows that $|E|=n-m-1$. Let
\[
w=x_1\prod_{\{i,j\}\in E}x_ix_j.
\]
Then $w\not\in J$ because $\deg w=n-m-1$. However we show that
$
x_iw\in J
$
 for $i=1,\ldots,n$, thereby proving that $\depth S/J=0$, as desired.

 Indeed, $x_1w-x_2w\in J$ and $x_2w\in J$. This implies that also $x_1w\in J$. Now let $i\neq 1,2$. If $i$ is a free vertex, let $(i_0,i_1,\ldots,i_k)$  be  the unique path  with  $i_0=i$, $i_{k-1}=2$, $i_k=1$ and $i_l\neq i_{l'}$ for $l\neq l'$. Then $\{i_l,i_{l+1}\}\in E$ for $l=1,\ldots,k-2$. Let
 \[
 E'=\{(i_l,i_{l+1})\:\; l=1,\ldots,k-2\}.
 \]
 Note that  $E'\subset E$,  and
 \[
 w=x_1\prod_{l=1}^{k-2}x_{i_l}x_{i_{l+1}}\prod_{\{i,j\}\in E\setminus E'}x_ix_j.
 \]
 Then
 \[
 x_iw=x_1x_2\prod_{l=0}^{k-3}x_{i_l}x_{i_{l+1}}\prod_{\{i,j\}\in E\setminus E'}x_ix_j\in I^{n-m}.
 \]
 In the above argument, if  $i$  is not a free vertex,  then $E'=\{(i_l,i_{l+1})\:\; l=0,\ldots,k-2\}$ and
 \[
 x_iw=(x_{i_0}x_{i_1})^2\prod_{l=2}^{k-1}x_{i_l}x_{i_{l+1}}\prod_{\{i,j\}\in E\setminus E'}x_ix_j\in I^{n-m}.
 \]
(b) We will show that $\depth S/I^{n-m-1}\geq 2$ if $G$ satisfies the assumptions of $(b)$. Together with (a) it then follows that $\dstab(I)=n-m$. For the proof of this statement we use the following lower bound for the depth of the powers of edge ideals of trees due to  Morey \cite[Corollary 3.7]{M}: let $G$ be tree with diameter $\delta$. The diameter is the maximal length of a path contained in $G$. Furthermore, let $q$ be the number of vertices $v$ of $G$ with the property that $v$ is not  a free vertex and there is at most one vertex in the neighborhood of $v$ which is not free. Then
\[
\depth S/I_G^k\geq \min\{\left\lceil\frac{\delta-k+q}{3}\right\rceil,1\}.
\]
We claim that $\delta-(n-m)+q=3$. The  Morey's lower bound gives $\depth S/I^{n-m-1}\geq 2$, as desired.

Let $P$ be a path in $G$ of length $\delta$. We prove the claim by induction on the number of vertices not belonging to $P$. If this number is zero, then $G=P$, and  the claim is obvious. Now assume that $G\neq P$. Let $v$ be a free vertex of $G$ of distance $1$ to $P$. Removing $v$ does not change   $\delta-(n-m)+q$. Let  $v$ be a free vertex of distance $2$ to $P$. Let  $w$  be the neighbor of $v$. Assume that $v_1,\ldots, v_r$ with $v=v_1$ are the free vertices whose neighbor is $w$. If $r>1$, removing $v_r$ does not change $\delta-(n-m)+q$. After these reductions we may assume  that each vertex $w$ of distance $1$ to $P$ has exactly $2$ neighbors.  It follows that $\delta-(n-m)+q=3$, as desired.
\end{proof}

\begin{Corollary}
\label{ab}
Given integers $a$ and $b$  with $1\leq a<b$. Then there exists a tree $G$ such that $\dstab(I_G)=a$ and $\ell(I_G)=b$.
\end{Corollary}

\begin{proof}
Let $P$ be a path of length $a$. We attach to one of the free vertices of $P$, $b-a$ leaves to obtain the graph $G$. Then $n=b+1$ and $m=(b-a)+1$. Therefore, $\ell(I_G)=n-1=b$ and $\dstab(I_G)=n-m=a$.
\end{proof}

\end{document}